\documentclass[11pt]{amsart}

\usepackage{amssymb}
\usepackage{enumitem}

\newtheorem{theorem}{{Theorem}}[section]
\newtheorem{defn}[theorem]{{Definition}}
\newtheorem{proposition}[theorem]{{Proposition}}
\newtheorem{isom.ext}[theorem]{{Trivial isometric extension}}
\newtheorem{definition}[theorem]{{Definition}}
\newtheorem{lemma}[theorem]{{Lemma}}
\newtheorem{remark}[theorem]{{Remark}}

\begin{document}

\title[Lattices in Splittable solvable Lie groups]{Splittable Lattices in the metabelian  solvable Lie group
$\mathbb{R}^n\rtimes\mathbb{R}^m$}

\author[Dali, Riahi]{B\'echir Dali $\& $ Moncef Riahi}
\address{ Department of Mathematics\\
Faculty of Sciences of Bizerte\\
University of Carthage\\
7021 Jarzouna, Bizerte - Tunisia\\}

\address{University of Carthage, Faculty of science of Bizerte, (UR17ES21),
"Dynamical Systems and their Applications", 7021, Jarzouna, Tunisia.}
\email{bechir.dali@fsb.ucar.tn, bechir.dali@fss.rnu.tn}
\email{moncef.riahi@fsb.ucar.tn, moncef.ipeib@gmail.com}




\maketitle

\begin{abstract}
The purpose of this note is  describe and classify the splittable lattices in the completely solvable metabelian Lie group (semidirect product of abelian vector groups) $G:=\mathbb{R}^n\rtimes_\eta\mathbb{R}^m$, where $\eta$ is the continuous representation of the topological additive abelian group  $\mathbb R^m$ in  $\mathbb R^n$ given by $\eta(t_1,\dots, t_m)=\exp(\sum_{j=1}^{m}t_j\Delta_j)$ with $(\Delta_j)_{1\leq j\leq m}$ is a set of pairwise commuting diagonal matrices in $\mathbb R^{n\times n}$.
\end{abstract}

\section{ Introduction}

In this paper, we shall use the following notational conventions. The symbol $\mathbb R^{m\times n}$ (resp.
$\mathbb Z^{m\times n}$) denotes the set of $m\times n$ real (resp. integer) matrices, $\mathbb R^n=\mathbb R^{n\times 1}$, $\mathbb Z^n=\mathbb Z^{n\times 1}$, $I_n$ is the $n\times n$ identity matrix, $0_n$ is the $n\times n$ null matrix.

By a lattice in a locally compact group $G$, we mean a discrete subgroup $\Gamma$ such that the homogeneous space $G/{\Gamma}$ carries a finite $G-$invariant Borelean measure. If in addition $G/{\Gamma}$ is compact,
$\Gamma$ is said to be a uniform lattice. A necessary condition for a group to contain a lattice is that the group must be unimodular. Recall that a Lie group $G$ is called unimodular if for all $X \in \frak g$
holds $tr (ad_X )= 0$, where $\frak g$ denotes the Lie algebra of $G$.
This allows for the easy construction of groups without lattices, for example the group of invertible upper triangular matrices or the affine groups. It is also not very hard to find unimodular groups without lattices.
For nilpotent Lie groups (which are unimodular) the theory simplifies much from the general case, and stays similar to the case of abelian groups. All lattices in a nilpotent Lie group are uniform, and if $N$ is a connected simply connected nilpotent Lie group (equivalently it does not contain a nontrivial compact subgroup) then a discrete subgroup is a lattice if and only if it is not contained in a proper connected subgroup \cite{MSR} (this generalises the fact that a discrete subgroup in a vector space is a lattice if and only if it spans the vector space).

A nilpotent Lie group $G$ contains a lattice if and only if the Lie algebra $\frak g$ of $G$ can be defined over the rationales. That is, if and
only if the Lie algebra $\frak g$ has a basis whose Lie
structure constants are integers.

While general lattices have a well-defined structure in Euclidean space and nilpotent groups, lattices in more general solvable groups are more complex and less rigid. Lattices in solvable Lie groups are much more difficult to handle than those in nilpotent Lie groups.
The criterion for nilpotent Lie groups to have a lattice given above does not apply to more general solvable Lie groups (see \cite{T}). It remains true that any lattice in a solvable Lie group is uniform and that lattices in solvable Lie groups are finitely generated (\cite{MSR}).\\

In \cite{MM}, the authors describe the set of lattices in a special class of solvable Lie groups $\mathbb R^n\rtimes\mathbb R$. In \cite{TsYama}, the authors consider solvable Lie groups which are isomorphic to unimodularizations
of products of affine groups, and it is shown that a lattice of such a Lie group is determined, up to commensurability, by a totally real algebraic number field.

In harmonic analysis, wavelets and lattices can be studied within the context of Lie groups by leveraging the group's inherent structure for translation and dilation, allowing for the construction of wavelet systems on stratified Lie groups and the investigation of discrete decompositions via frames on function spaces like $L^2(G)$. Lattices, which are discrete subgroups, play a crucial role in defining the sampling points for these wavelets and ensuring the existence and properties of frames. Research in this area explores how to build wavelet frames on Lie groups and their subgroups, adapting concepts from Euclidean spaces to the more complex geometric settings of these groups. First identify a lattice $\Gamma$ within a simply connected solvable Lie group $G$. Then, use the structure of the group, often expressible as a semi-direct product $G\cong N\rtimes _{\phi }S$, and the lattice to define a set of operators for scaling and translation.
This involves creating a basis of functions by applying these operators, derived from the lattice elements, to a fundamental wavelet function or a scaling function. Wavelets can be constructed on the homogeneous spaces $G/\Gamma $ that are formed by taking a quotient of a Lie group $G$ by a lattice. The commensurability of lattices can be used to construct wavelets by leveraging the fact that the lattices are "closely related". This is especially true when using the decomposition of a split solvable Lie group and the commensurability of  lattices provides the framework for building a complete wavelet system.

 In this paper we are concerned with the characterization of lattices in the class of metabelian solvable (non nilpotent) Lie groups of the form $G=\mathbb R^n\rtimes_\eta\mathbb R^m$ where $\eta$ is a continuous representation of the topological additive group $\mathbb R^m$ in $SL_n(\mathbb R)$.




The paper is organized as follows: in Section 2, we recall some definitions and general properties of the lattices in solvable Lie groups. We recall also some basic tools which will be useful for the rest of the paper. In Section 3, we describe the splittable lattices in the completely solvable Lie group $G=\mathbb R^n\rtimes_\eta\mathbb R^m$ where $\eta$ is a continuous representation of the topological additive group $\mathbb R^m$ given by
$$
\eta\left(
\begin{array}{c}
t_1 \\
\vdots \\
t_m \\
\end{array}
\right)
=\exp{\left(\sum_{i=1}^{m}t_i\Delta_i\right)},\quad \left(
\begin{array}{c}
t_1 \\
\vdots \\
t_m \\
\end{array}
\right)\in \mathbb R^m
$$
with $\Delta_1,\dots, \Delta_m\in\mathbb R^{n\times n}$ satisfying certain properties. We also investigate the commensurability of two splittable lattices in $G$. Finally, in Section 4, we conclude the paper by giving examples of lattices in $\mathbb R^n\rtimes_\eta\mathbb R$ for $n=2,3$.

\section{Preliminaries}

\subsection{Solvmanifolds} A solvmanifold is a compact homogeneous space
$G/\Gamma$, where $G$ is a connected and simply-connected solvable Lie group and $\Gamma$ a lattice
in $G$, i.e. a discrete co-compact subgroup.
Every connected and simply connected solvable Lie group is diffeomorphic to
$\mathbb R^m$ (see e.g. \cite{VS}), hence solvmanifolds are aspherical and their fundamental group
is isomorphic to the considered lattice.
Unfortunately, there is no simple criterion for the existence of a lattice in a connected and simply-connected solvable Lie group. We shall quote some necessary
criteria.

\begin{proposition} (\cite{M}). If a connected and simply-connected solvable Lie group admits a lattice then it is unimodular.
\end{proposition}

The following theorems  are fundamental
to the study of lattices in arbitrary solvable Lie groups. Proved first
by Mostow, it throws a footbridge between lattices in arbitrary solvable Lie
groups and lattices in nilpotent Lie groups.
In the statement of this theorem a nilradical in a connected solvable Lie
group $G$ is the largest connected nilpotent normal subgroup in $G$.

\begin{theorem}(Mostow 1964, Auslander 1973. Raghunathan 1972). If $\Gamma$
is a lattice in a connected solvable Lie group $G$, and $N $ the nilradical of the
group $G$, then $\Gamma\cap N $ is a lattice in $N$ .
\end{theorem}

\begin{theorem}(\cite{Mo, OT} ).
Let $G/\Gamma$ be a solvmanifold that is not
a nilmanifold and denote by $N$ the nilradical of $G$.
Then $\Gamma_N:= \Gamma\cap N$ is a lattice in $N$, $\Gamma N = N\Gamma$ is a closed subgroup in $ G $ and
$G/(N\Gamma)$ is a torus. Therefore, $G/\Gamma $ can be naturally fibred over a non-trivial torus
with a nilmanifold as fiber:
$$
N/\Gamma N = (N\Gamma)/\Gamma \rightarrow G/\Gamma \rightarrow G/(N\Gamma ) = \mathbb T^k
$$
his bundle is called the Mostow bundle.
\end{theorem}

\begin{theorem}(\cite{MSR})
(Moore 1963 )\label{Mo-Ragh} Suppose that $\Gamma$ is a
discrete subgroup in a solvable Lie group $G$. Then the following assertions are equivalent:
\begin{enumerate}[label=\roman*)]
\item $\Gamma$ is a lattice in $G$;
\item $\Gamma$ is a uniform lattice in $G$.
\end{enumerate}

\end{theorem}

In distinction from the nilpotent case, criteria for the existence of a lattice
in connected and simply-connected solvable Lie groups have rather cumbersome
formulations
%
%
%
%
%
%
%
%

Let $G$ be a simply-connected solvable Lie group with nilradical $ N $, then $G$ satisfies the exact sequence
$$
1\rightarrow N\rightarrow G\rightarrow\mathbb R^s\rightarrow 1
$$
We say that $G$ is splittable if the short exact sequence splits, that is, there is a right inverse homomorphism of the projection
$G \rightarrow \mathbb R^s$. This condition is equivalent to the existence of a homomorphism
$\pi : \mathbb R^s \rightarrow Aut(N)$ such that $G$ is isomorphic to the semi-direct product $N\rtimes_\pi\mathbb R^s$.
%

\subsection{Semi-direct product of groups}\label{semidirect}

For given (Lie) groups $N, H$ and a (smooth) action $\mu: H \times N \rightarrow N$ by (Lie)
group automorphisms, one defines the semidirect product of $ N$ and $H$ via $\mu$ as
the (Lie) group $G=N\rtimes_\mu H$ with underlying set (manifold) $ N\times H$ and group structure
defined as follows:
$$
\forall (n_i,h_i)\in N\times H,\quad i=1,2\quad (n_1, h_1)(n_2, h_2) =(\mu(h_1,n_1)n_2,h_1h_2).
$$
Note that for $(n,h)\in N\rtimes_\mu H$ we have $(n,h)^{-1}=(\mu(h^{-1},n^{-1}),h^{-1})$.
Equivalently, if $\alpha $ is a smooth morphism
$$
\alpha: H\rightarrow Aut(N),\qquad h\mapsto \alpha(h):=\mu(h,~),
$$
the (Lie) group $G:=N\rtimes_\mu H$ is defined as:
$$
(n_1, h_1)(n_2, h_2) =(n_1\alpha(h_1)n_2,h_1h_2).
$$
Denote the Lie algebras of $N$ and $H$ by
$\frak n$ and $\frak h$ and let $ \phi := (d_{e_N}\mu_1: \frak n\rightarrow {\partial}(\frak h)$, where $\mu_1 : N\rightarrow Aut(\frak h)$ is given by
$$
\mu_1(h) = d_{e_N} \mu(h,~) = Ad^{N\rtimes_\mu H}_h.
$$
The Lie algebra of $G:=N \rtimes_\mu H$ is $\frak g:=\frak n\rtimes_\phi\frak h$ is called semidirect product of $\frak g$ and $\frak h$ via $\phi$. Its underlying vector space is $\frak n\times\frak h$ and the Lie bracket for $(X_i,Y_i)\in\frak n\times \frak h,i=1,2$ is given by
$$
[(X_1,Y_1),(X_2,Y_2)]=([X_1,X_2]_{\frak n}+\phi(Y_1)X_2-\phi(Y_2)X_1, [Y_1,Y_2]_{\frak h}) .
$$

\begin{defn}
Let $G$ be a Lie group with Lie algebra $\frak g$.
\begin{itemize}
\item [(i)] $ G$ and $\frak g$ are called completely solvable if the linear map $ad_X : \frak g \rightarrow \frak g$ has only
real roots for all $X\in\frak g$.

\item [(ii)] If $ G$ is simply-connected and $\exp: \frak g \rightarrow G$ is a diffeomorphism, then $G$ is
called exponential.

\end{itemize}

\end{defn}
A nilpotent Lie group or algebra is completely solvable, and it is easy to see that
completely solvable Lie groups or algebras are solvable. Moreover, any simply-connected completely solvable Lie group is exponential, and any exponential Lie group is solvable.
A connected and simply-connected solvable Lie group $ G$ with
Lie algebra $\frak g$ is exponential if and only if the linear map $ad_X: \frak g \rightarrow \frak g $ has no purely
imaginary roots for all $X\in\frak g$ (for more details see \cite{BCDLRRV, VGS, VS}).

\begin{remark}\label{nilradical}
Let $G$ be a solvable Lie group and $N$ its nilradical.
Then $\dim N \geq \frac{1}{2} \dim G$ (see \cite{BCDLRRV}).
\end{remark}

Let $(\Delta_j)_{1\leq j\leq m}\subset\mathbb R^{n\times n}$ be a set of pairwise commuting matrices and let  $G:=\mathbb R^n\rtimes_\eta\mathbb R^m$ be the group endowed with the law
$$
\begin{array}{ccc}
(x, t) (y, s)& = & (x+\eta( t)y, t+ s) \\
& = & (x+e^{ t\cdot{\Delta}}y, t+ s)
\end{array}
$$
where
$$
{t\cdot{ \Delta}}= \sum_{i=1}^{m}t_i\Delta_i,\quad {\Delta}=(\Delta_1,\dots,\Delta_m),\,\, t=\left(
\begin{array}{c}
t_1 \\
\vdots \\
t_m \\
\end{array}
\right)\in\mathbb R^m,$$
and $ \eta$ is the continuous representation of the topological additive (abelian) group $\mathbb{R}^m$ given by
$$\eta: \mathbb{R}^m\rightarrow GL_n(\mathbb{R}),\quad t\mapsto \eta( t)=e^{t\cdot { \Delta}},$$
here $e^{t\cdot {\Delta}}$ is the matrix exponential of the matrix $t\cdot\Delta$. The inverse of $(x,t)\in G$ is given by
$$ (x,t)^{-1}=(-e^{-t\cdot \Delta}x,-t).
$$
The Lie algebra of $G$ is $\frak g=\mathbb R^n\oplus\mathbb R^m$ and is equipped with the Lie bracket
\[
[(X, t), (Y, s )]=( (t\cdot\Delta) Y- (s\cdot\Delta)X,0),\quad X,Y\in\mathbb R^n, \quad t, s\in\mathbb R^m.
\]

%
%
%
%
%

From now on we assume that $(\Delta_i)_{1\leq i\leq m}$ is a set of linearly independent diagonal non singular traceless matrices with
$$
\Delta_{i}=\mathcal{D}(d_{1}^{(i)},\ldots,d_{n}^{(i)})
:=\left(\begin{array}{ccc}
d_{1}^{(i)} & & 0 \\
& \ddots & \\
0 & & d_{n}^{(i)} \\
\end{array}
\right)\in GL_{n}(\mathbb{R})\cap \frak{sl}_n(\mathbb R),$$
such that
$$
d_{k}^{(i)}\neq d_{j}^{(i)},\quad i=1,\dots, m,\quad k\neq j=1,\dots,n.
$$

Note that the matrix
$$\Omega=\left(
\begin{array}{ccc}
d_{1}^{(1)} & \ldots & d_{1}^{(m)} \\
\vdots & & \vdots \\
d_{n}^{(1)} & \ldots & d_{n}^{(m)} \\
\end{array}
\right)$$
is injective and hence it has a left inverse which is $ \left({}^{t}\Omega\Omega\right)^{-1}{}^{t}\Omega$.
The Lie algebra $\frak g$ of $G$ is completely solvable since $spec(\mathrm{ad}_{(X, t)})\subset\mathbb R$ for all $(X, t)\in\frak g$ (for more details see \cite{CFT}), and hence the exponential mapping $\exp:\frak g\rightarrow G$ is a $\mathcal C^\infty$ diffeomorphism. The nilradical of $G$ is $\mathbb R^n\times\{0\}$ and by Remark \ref{nilradical}, one has $1\leq m\leq n$. On other hand, the Lie algebra $\frak g$ can be viewed as the Lie algebra with basis $(X_1,\dots,X_n,T_1,\dots,T_m)$ and non trivial brackets
$$
[T_i,X_j]=d_j^{(i)}X_j,\quad i,=1,\dots,m,\quad j=1,\dots,n.
$$
The Lie group $G$ is unimodular since $|\det \mathrm{\mathrm{Ad}}_{(x,t)}|=1$ for all $(x,t)\in G$, and it can be realized in $GL_{n+1}(\mathbb R)$ as the group of upper triangular invertible matrices:
$$
G=\left\{\left( \begin{array}{cccc}
e^{\sum_{i=1}^{m}t_id_1^{(i)}} & 0 &\cdots& x_1 \\
0 & \ddots & 0 & \vdots \\
\vdots & 0 & e^{\sum_{i=1}^{m}t_id_n^{(i)}} & x_n \\
0 & \ldots & 0& 1 \\
\end{array}
\right), \quad t_1,\dots, t_m,x_1,\dots,x_n\in\mathbb R
\right\}
$$
and the corresponding Lie algebra
$$
\frak g=\left\{\left( \begin{array}{cccc}
\sum_{i=1}^{m}t_id_1^{(i)} & 0 &\cdots& x_1 \\
0 & \ddots & 0 & \vdots \\
\vdots & 0 & \sum_{i=1}^{m}t_id_n^{(i)} & x_n \\
0 & \ldots & 0& 0 \\
\end{array}
\right), \quad t_1,\dots, t_m,x_1,\dots,x_n\in\mathbb R
\right\}
$$

\section{Lattices in $G$}

In this section, we are concerned with the description of splittable lattices in $G$.

\begin{definition}\label{Gcompatible}
Let $(\sigma,\rho)\in GL_{n}(\mathbb{R})\times GL_{m}(\mathbb{R})$, we shall say that the pair
$(\sigma,\rho)$ is $G-$compatible if
$$\sigma \exp\left(\rho^{(j)}\cdot\Delta\right)\sigma^{-1}\in SL_n(\mathbb Z)$$
for all $j=1,\dots, m$, where $\rho^{(j)}$ denotes the $j^{th}$ column of $\rho$.
\end{definition}
\begin{remark}

\begin{itemize}
\item If $(\sigma,\rho)$ is a $G-$compatible pair, then for any $(\lambda ,q)\in \mathbb R^\star\times\mathbb Z^\star$, the pair $(\lambda \sigma , q\rho)$ is also $G-$compatible.
\item If $(\sigma,\rho)$ is a $G-$compatible pair, then for any $k\in\mathbb Z^m$ one has
$$\sigma \exp\left((\rho\, k)\cdot\Delta\right)\sigma^{-1}\in SL_n(\mathbb Z).$$
\end{itemize}
\end{remark}

\begin{theorem}\label{Glattice}
Let $(\sigma,\rho)$ be a $G$-compatible pair as in Definition \ref{Gcompatible} and denote by $L_{(\sigma,\rho)}:=\sigma^{-1}\mathbb{Z}^{n}\rtimes_{\eta}\rho\mathbb{Z}^{m}$. Then $L_{(\sigma,\rho)}$ is a lattice in $G$.
\end{theorem}
\begin{proof}
Let $(\sigma^{-1}v,\rho(p)), (\sigma^{-1}w,\rho(q))\in L_{(\sigma,\rho)}$ with $\rho:=(\rho_{i,j})_{1\leq i,j\leq m}$.
Using the inverse property and applying the group multiplication rule, we compute
\begin{align*}
(\sigma^{-1}v,\rho(p)) (\sigma^{-1}w,\rho(q))^{-1}= & (\sigma^{-1}v,\rho(p))(-e^{-\rho(q)\cdot\Delta}\sigma^{-1}w,-\rho(q)) \\
= & (\sigma^{-1}v-e^{\rho(p)\cdot\Delta} e^{-\rho(q)\cdot\Delta}\sigma^{-1}w,\rho(p)-\rho(q)) \\
= & (\sigma^{-1}v-e^{\rho(p)\cdot\Delta} e^{-\rho(q)\cdot\Delta}\sigma^{-1}w,\rho(p)-\rho(q)) \\
= & (\sigma^{-1}v- e^{\rho(p-q)\cdot\Delta}\sigma^{-1}w),\rho(p-q))\\
= & (\sigma^{-1}(v-\sigma e^{\rho(p-q)\cdot\Delta}\sigma^{-1}w),\rho(p-q)).
\end{align*}

Thus
$$
(\sigma^{-1}v,\rho(p)) (\sigma^{-1}w,\rho(q))^{-1}\in L_{(\sigma,\rho)}\quad\forall v,w\in\mathbb Z^n, p,q\in\mathbb Z^m.
$$
Thus $L_{(\sigma, \rho)}$ is a discrete subgroup of $G$. We show now that $L_{(\sigma, \rho)}$ is co-compact.

Precisely, we will prove that
$$G/L_{(\sigma, \rho)}=\left\{\overline{(x,t)}:~\ (x,t)\in (\sigma^{-1}[0,1)^{n})\times(\rho[0,1)^{m})\right\},$$
where $\overline{(x,t)}$ denotes the class of $(x,t)$ modulo $L_{(\sigma, \rho)}$.\\
Let $(x,t)$ be an arbitrary element of $G$. Take $y\in[0,1)^{m}$ and $k\in\mathbb{Z}^{m}$ such that
\begin{equation*}
\rho^{-1}t=y+k.
\end{equation*}
Further, let $z\in[0,1)^{n}, l\in\mathbb{Z}^{n}$ be such that
\begin{equation*}
\sigma e^{(-\rho k)\cdot\Delta}x=z+l.
\end{equation*}
Let's rewrite the expression of $(x,t)(\sigma^{-1}z,\rho y)^{-1}$.
We have
\begin{align*}
(x,t)(\sigma^{-1}z,\rho y)^{-1}=&(x,t)(-e^{-\rho y\cdot\Delta}\sigma^{-1}z,-\rho y) \\
=&(x-e^{(t-\rho y)\cdot\Delta}\sigma^{-1}z,\rho(\rho^{-1}t- y)).
\end{align*}
Since $t-\rho y=\rho k$, this reduces to
$$(x,t)(\sigma^{-1}z,\rho y)^{-1}=(x-e^{(\rho k)\cdot \Delta}\sigma^{-1}z,\rho k).$$
Using the expression $x=e^{(\rho k)\cdot\Delta}\sigma^{-1}(z+l)$, we have
$$(x,t)(\sigma^{-1}z,\rho y)^{-1}=(e^{(\rho k)\cdot \Delta}\sigma^{-1}l,\rho k).$$
The right-hand side can be rewritten as:
$$(\sigma^{-1}\sigma e^{(\rho k)\cdot\Delta}\sigma^{-1}l,\rho k)\in L_{(\sigma, \rho)},$$
since $l\in\mathbb Z^{n}$ and $k\in\mathbb Z^{m}$.
It follows that $$\overline{(x,t)}=\overline{(\sigma^{-1}z,\rho y)}$$

This shows that
$$
G/L_{(\sigma, \rho)}=\left\{\overline{(x,t)}:~\ (x,t)\in (\sigma^{-1}[0,1)^{n})\times(\rho[0,1)^{m})\right\}$$
and hence $L_{(\sigma, \rho)}$ is co-compact. Therefore, $L_{(\sigma, \rho)}$ is a lattice in $G$.
\end{proof}

Let $S_{n}$ be the group of permutations of the set $\{1,\ldots,n\}$. If $\tau\in S_{n}$ we denote by $P_{\tau}$ the permutation matrix obtained from the identity matrix $I_{n}$ by changing, for all $1\leq i\leq n$, the $i^{th}$ row by the $\tau^{-1}(i)^{th}$ row.
\begin{proposition}\label{lem}
Let $\phi$ be an endomorphism of the group $G$ and let $\alpha:\mathbb{R}^{n}\longrightarrow\mathbb{R}^{n}$, $\beta:\mathbb{R}^{m}\longrightarrow\mathbb{R}^{n}$, $\gamma:\mathbb{R}^{n}\longrightarrow\mathbb{R}^{m}$ and $\delta:\mathbb{R}^{m}\longrightarrow\mathbb{R}^{m}$ be defined as follows
$$
\begin{array}{cc}
\alpha=\pi_{{}_{\mathbb{R}^{n}}}\circ\phi\circ j_{{}_{\mathbb{R}^{n}}} & \beta=\pi_{{}_{\mathbb{R}^{n}}}\circ\phi\circ j_{{}_{\mathbb{R}^{m}}} \\
\gamma=\pi_{{}_{\mathbb{R}^{m}}}\circ\phi\circ j_{{}_{\mathbb{R}^{n}}} & \delta=\pi_{{}_{\mathbb{R}^{m}}}\circ\phi\circ j_{{}_{\mathbb{R}^{m}}},
\end{array}
$$
where $\pi$ (resp. $j$) denotes the projection (resp. the injection) onto (resp. from) the indicated parts of the semidirect product.\\
Then $\phi$ is an automorphism of $G$ if and only if the following conditions hold:
\begin{enumerate}[label=\arabic*)]
\item $\alpha=P_{\tau}\mathcal{D}(c_{1},\ldots,c_{n})$
for some non zero real numbers $c_{1},\ldots,c_{n}$, $\tau\in S_{n}$ and $P_\tau$ is the permutation matrix corresponding to $\tau$.
\item $\gamma=0$,
\item $\delta=\left({}^{t}\Omega \Omega \right)^{-1}{}^{t}\Omega P_{\tau}\Omega $,
\item $\beta(t)=\sum_{k=0}^{\infty}\frac{(\delta(t)\cdot\Delta)^k}{(k+1)!}Ut$, for some $U\in\mathbb R^{n\times m}$.
\end{enumerate}
\end{proposition}
\begin{proof}
Let $\phi$ be a morphism of $G$, thus we can write
$$\phi(x,0)=(\alpha(x),\gamma(x)) \text { and } \phi(0,t)= (\beta(t),\delta(t)),$$
with
$$
\begin{array}{cc}
\alpha=\pi_{{}_{\mathbb{R}^{n}}}\circ\phi\circ j_{{}_{\mathbb{R}^{n}}} & \beta=\pi_{{}_{\mathbb{R}^{n}}}\circ\phi\circ j_{{}_{\mathbb{R}^{m}}} \\
\gamma=\pi_{{}_{\mathbb{R}^{m}}}\circ\phi\circ j_{{}_{\mathbb{R}^{n}}} & \delta=\pi_{{}_{\mathbb{R}^{m}}}\circ\phi\circ j_{{}_{\mathbb{R}^{m}}},
\end{array}
$$
Now since $(x,t)=(x,0)(0,t)$ and $\phi$ is a morphism of $G$, it derives that
\begin{equation}\label{law}\phi(x,t)=(\alpha(x)+e^{(\gamma(x)\cdot\Delta)}\beta(t),\gamma(x)+\delta(t)).
\end{equation}
On the other hand, since $(0,t)(0,s)=(0,t+s)$ and $\phi\in Hom (G)$ therefore

\begin{equation}\label{betadelta}
\left\{
\begin{array}{lll}
\delta(t+s) &= & \delta(t)+\delta(s) \\
\beta(t+s)&= & \beta(t)+e^{\delta(t)\cdot\Delta}\beta(s)
\end{array}
\right.
\end{equation}

From this, we see that $\delta\in End(\mathbb R^m)$.
Similarly, since $(x,0)(y,0)=(x+y,0)$ we get
\begin{equation}\label{betadelta}
\left\{
\begin{array}{lll}
\gamma(x+y) &= & \gamma(x)+\gamma(y) \\
\alpha(x+y)&= & \alpha(x)+e^{\gamma(x)\cdot\Delta}\gamma(y)
\end{array}
\right.
\end{equation}
From this, it derives that $\gamma\in Hom(\mathbb R^n,\mathbb R^m)$.
At the end of this step, we can write
\begin{align*}
\phi(x,t)&=\phi(x,0)\phi(0,t)\\
&=(\alpha(x)+e^{\gamma(x)\cdot\Delta}\beta(t),\gamma(x)+\delta(t)).
\end{align*}

Suppose that $\phi$ is an automorphism of $G$. Note that $\mathbb{R}^{n}\times\{0_{\mathbb{R}^{m}}\}$ is characteristic in $G$ since it is the nilradical of $G$, therefore $\alpha$ is an automorphism of $\mathbb{R}^{n}$ and $\gamma=0$. Furthermore, $\delta$ is an automorphism of $\mathbb{R}^{m}$.

Thus for every $(x,t)\in G$ we have $\phi(x,t)=(\alpha(x)+\beta(t),\delta(t))$. Therefore
\begin{align*}
\phi((x,t)(y,s))=&\phi(x+\eta(t)y,t+s)\\
=&(\alpha(x)+\alpha(\eta(t)y)+\beta(t+s),\delta(t+s))
\end{align*}
On other hand one has
\begin{align*}
\phi(x,t)\phi(y,s)=&(\alpha(x)+\beta(t),\delta(t))(\alpha(y)+\beta(s),\delta(s))\\
=&(\alpha(x)+\beta(t)+\eta(\delta(t))(\alpha(y)+\beta(s)),\delta(t)+\delta(s)).
\end{align*}
The homomorphism condition implies

\begin{align}\label{eqq1}
\alpha(\eta(t)y)+\beta(t+s)
=&\beta(t)+\eta(\delta(t))(\alpha(y))+\eta(\delta(t))(\beta(s)).
\end{align}
Setting $t=0$ in (\ref{eqq1}), we get (since $\eta(0)=Id_{{}_{\mathbb{R}^{n}}}$ and $\delta(0)=0$)
$$\beta(0)=0.$$
Setting $s=0$ in (\ref{eqq1}), we get
\begin{align}\label{eqq2}
\alpha\eta(t)=&\eta(\delta(t))\alpha,\quad \text{for all } t\in\mathbb R^m.
\end{align}
Then the equality (\ref{eqq1}) becomes
\begin{align}\label{eqq3}
\beta(t+s)=&\beta(t)+\eta(\delta(t))(\beta(s)), \quad \text{for all } t, s\in\mathbb R^m.
\end{align}
Setting $t=e_{k}=(0,\ldots,1,\ldots,0)$ (the $k^{th}$ element of the standard basis of $\mathbb{R}^{m}$) in (\ref{eqq2}), we get
\begin{align}\label{eqq4}
\alpha e^{\Delta_{k}}=& e^{\delta^{(k)}\cdot{\Delta}}\alpha,\quad(\text{where }\delta^{(k)}=\delta e_{k})
\end{align}
it follows that
\begin{align*}
\alpha \exp\left(\Delta_{k}\right)\alpha^{-1}=& \exp\left(\sum_{l=1}^{m}\delta{{}_{lk}}\Delta_{l}\right),
\end{align*}
hence for every $1\leq i\leq n$,
\begin{align}\label{eqq5}
\sum_{l=1}^{m}d_{i}^{(l)}\delta{{}_{lk}}=&d_{\tau(i)}^{(k)},
\end{align}
for some $\tau\in S_{n}$.
In fact, the equality (\ref{eqq4}) becomes
\begin{align}\label{eqq6}
\alpha e^{\Delta_{k}}=& e^{\Delta_{k,\tau}}\alpha
\end{align}
where $$\Delta_{k,\tau}=\mathcal{D}(d^{(k)}_{\tau(1)},\ldots,d^{(k)}_{\tau(n)}).$$
Let $(\varepsilon_{1},\ldots,\varepsilon_{n})$ the standard basis of $\mathbb{R}^{n}$. From (\ref{eqq6}) we get
$$e^{\Delta_{k,\tau}}\alpha \varepsilon_{i}=e^{d^{(k)}_{i}}\alpha \varepsilon_{i},\quad\text{ for all } i=1,\dots, n.$$
It follows that $\alpha \varepsilon_{i}$ is an eigenvector of $e^{\Delta_{k,\tau}}$ with respect to the eigenvalue $e^{d^{(k)}_{i}}$, so, since the $d^{(k)}_{l},\ 1\leq l\leq n$ are pairwise distinct,
\begin{align}\label{eqq7}
\alpha \varepsilon_{i}=c_{i}\varepsilon_{\tau^{-1}(i)}.
\end{align}
for some non zero real number $c_i$. Hence $\tau$ is independent of $k$ and
$$\alpha_{ij}=\left\{\begin{array}{ccc}
c_{\tau(i)} & if & j=\tau(i) \\
0 & if & j\neq\tau(i).
\end{array}
\right.
$$
Here the reals $c_1,\dots, c_n$ are non zero since $\alpha\in Aut(\mathbb R^{n})$. Thus $$\alpha=P_{\tau}\mathcal{D}(c_{1},\ldots,c_{n}).$$
From (\ref{eqq5}), if $A$ denotes the matrix of $\alpha$ with respect to the standard basis of $\mathbb R^n$, we get
$$\Omega A=P_{\tau}\Omega, $$
and hence $$A=\left({}^{t}\Omega \Omega \right)^{-1}{}^{t}\Omega P_{\tau}\Omega .$$
Recall that $\beta$ satisfies
$$
\beta(t+s)=\beta(t)+\eta(\delta(t))(\beta(s)),\quad \text{for all } t, s\in\mathbb R^m,
$$
that is,
\begin{align*}
\beta(t+s)-\beta(t) =& \eta(\delta(t))(\beta(s)) \\
=& \eta(\delta(t))(\beta(s)-\beta(0))
\end{align*}
Therefore we get
\begin{align*}
D\beta(t)=&\eta(\delta(t))D\beta(0)\\
=&\exp\left(\delta(t)\cdot\Delta\right)D\beta(0),\quad t\in\mathbb{R}^{m}.
\end{align*}
On the other hand, the Taylor formula gives
$$
\beta(t)=\beta(0)+\int_0^1D\beta(\lambda t)t{\mathrm d}\lambda=\left(\int_0^1 e^{\lambda\delta(t)\cdot\Delta}{\mathrm d}\lambda\right)D\beta(0)t.
$$
Now we shall compute $\int_0^1 e^{\lambda\delta(t)\cdot\Delta}d\lambda$, to this end recall that
$$
e^{\lambda\delta(t)\cdot\Delta}=\sum_{k=0}^{\infty}\frac{\lambda^k}{k!}(\delta(t)\cdot\Delta)^k
$$ and hence
$$
\int_0^1 e^{\lambda\delta(t)\cdot\Delta}\rm d\lambda=\sum_{k=0}^{\infty}\left(\int_0^1\frac{\lambda^k}{k!}
\rm d\lambda\right)(\delta(t)\cdot\Delta)^k=\sum_{k=0}^{\infty}
\frac{(\delta(t)\cdot\Delta)^k}{(k+1)!}.
$$
Then
$$
\beta(t)=\sum_{k=0}^{\infty}\frac{(\delta(t)\cdot\Delta)^k}{(k+1)!}Ut,
$$
with $U=D\beta(0)\in\mathbb R^{n\times m}$.

The converse follows from the preceding calculations.
\end{proof}
\begin{theorem}
Let $(\sigma,\rho)$ and $(\nu,\varrho)$ be two $G$-compatible pairs.
The lattices $L_{(\sigma,\rho)}=\sigma^{-1}\mathbb{Z}^{n}\rtimes_{\eta}\rho\mathbb{Z}^{m}$ and $L_{(\nu,\varrho)}=\nu^{-1}\mathbb{Z}^{n}\rtimes_{\eta}\varrho\mathbb{Z}^{m}$ differ by an automorphism $\phi:=(\alpha+\beta,\delta)$ of $G$ if and only if
$$
\nu\alpha\sigma^{-1}\in GL_{n}(\mathbb{Z}) \quad \text{and} \quad
\varrho^{-1}\delta\rho\in GL_{m}(\mathbb{Z}).
$$

\end{theorem}
\begin{proof}Assume that there is an automorphism $\phi$ of $G$ with $\phi=(\alpha+\beta,\delta)$ and such that
$\phi(L_{(\sigma,\rho)})=L_{(\nu,\varrho)}$.
Then
$$\alpha(\sigma^{-1}\mathbb{Z}^{n})=\nu^{-1}\mathbb{Z}^{n}\quad\text{and}\quad\delta(\rho\mathbb{Z}^{m})
=\varrho\mathbb{Z}^{m},$$
it follows that
$$\nu\alpha\sigma^{-1}\in SL_{n}(\mathbb{Z})\quad\text{and}\quad\varrho^{-1}\delta\rho\in SL_{m}(\mathbb{Z})$$
Conversely, suppose that there exist $\phi:=(\alpha,\delta)\in Aut(G)$ such that
$$
\nu\alpha\sigma^{-1}\in SL_{n}(\mathbb{Z}) \quad
\text{and}\quad
\varrho^{-1}\delta\rho\in SL_{m}(\mathbb{Z})
$$
Then we have
\begin{align*}
\alpha(\sigma^{-1}\mathbb{Z}^{n})= & \nu^{-1}(\nu\alpha\sigma^{-1}\mathbb{Z}^{n}) \\
= & \nu^{-1}\mathbb{Z}^{n},
\end{align*}
and
\begin{align*}
\delta(\rho\mathbb{Z}^{m})= & \varrho(\varrho^{(-1)} \delta\rho \mathbb{Z}^{m}) \\
= &\varrho\mathbb{Z}^{m}.
\end{align*}
Thus
$$\phi(L_{(\sigma,\rho)})=L_{(\nu,\varrho)}$$
\end{proof}

Let $\Gamma$ and $\Gamma'$ be lattices of $G$. Recall that they are said to be commensurable if $[\Gamma :\Gamma \cap\Gamma']<\infty$
and $[\Gamma' : \Gamma \cap\Gamma'] <\infty$ (see e.g., \cite{VGS}).
\begin{proposition}\label{Gcommensurable1}
Let $(\sigma ,\rho)$ and $(\nu,\varrho)$ be two $ G$-compatible pairs.
The lattices $L_{(\sigma ,\rho)}=\sigma^{-1}\mathbb{Z}^{n}\rtimes_\eta\rho\mathbb{Z}^{m}$ and $L_{(\nu,\varrho)}=\nu^{-1}\mathbb{Z}^{n}\rtimes_\eta\varrho\mathbb{Z}^{m}$ are commensurables if and only if
$$\left(\mathrm{rank}_{\mathbb{Z}}(\sigma^{-1}\mathbb{Z}^{n}\cap\nu^{-1}\mathbb{Z}^{n}),
\mathrm{rank}_{\mathbb{Z}}(\rho\mathbb{Z}^{m}\cap\varrho\mathbb{Z}^{m})\right)=(n,m).$$
\end{proposition}
\begin{proof}
Recall that the intersection $L_{(\sigma,\rho)}\cap L_{(\nu,\varrho)}$ is given by
$$L_{(\sigma,\rho)}\cap L_{(\nu,\varrho)}=(\sigma^{-1}\mathbb{Z}^{n}\cap\nu^{-1}\mathbb{Z}^{n})\times
(\rho\mathbb{Z}^{m}\cap\varrho\mathbb{Z}^{m}).$$

Let $$n_{0}=\mathrm{rank}_{\mathbb{Z}}(\sigma^{-1}\mathbb{Z}^{n}\cap\nu^{-1}\mathbb{Z}^{n}),\quad
m_{0}=\mathrm{rank}_{\mathbb{Z}}(\rho\mathbb{Z}^{m}\cap\varrho\mathbb{Z}^{m}).$$
Let $v_{1},\ldots,v_{n}\in\sigma^{-1}\mathbb{Z}^{n}$ and $p_{1},\ldots,p_{n_{0}}\in\mathbb{N}$ such that $(v_{1},\ldots,v_{n})$ is a basis of $\sigma^{-1}\mathbb{Z}^{n}$ and $(p_{1}v_1,\ldots,p_{n_{0}}v_{n_{0}})$ is a basis of $(\sigma^{-1}\mathbb{Z}^{n}\cap\nu^{-1}\mathbb{Z}^{n})$.
Let $w_{1},\ldots,w_{m}\in\rho\mathbb{Z}^{m}$ and $q_{1},\ldots,q_{m_{0}}\in\mathbb{N}$ such that $(w_{1},\ldots,w_{m})$ is a basis of $\rho\mathbb{Z}^{m}$ and $(q_{1}w_{1},\ldots,q_{m_{0}}w_{m_{0}})$ is a basis of
$(\rho\mathbb{Z}^{m}\cap\varrho\mathbb{Z}^{m})$.

Suppose that $n_{0}<n$. In this case we get
$$\mathbb{Z}v_{n}\cap(\sigma^{-1}\mathbb{Z}^{n}\cap \nu^{-1}\mathbb{Z}^{n})=\{0\},
$$ and it follows that $\{\overline{(kv_{n},0)},\ k\in\mathbb{Z}\}$ is an infinite subset of $L_{(\sigma,\rho)}/L_{(\sigma,\rho)}\cap L_{(\nu,\varrho)}$, thus $L_{(\sigma,\rho)}$ and $L_{(\nu,\varrho)}$ are non commensurable, here $\overline{(kv_{n},0)}$ is the class of $(kv_{n},0)$ modulo $L_{(\sigma,\rho)}\cap L_{(\nu,\varrho)}$.

Similarly, if $m_{0}<m$ then $\{\overline{(0,lw_{m})},\ l\in\mathbb{Z}\}$ is an infinite subset of $L_{(\sigma,\rho)}/L_{(\sigma,\rho)}\cap L_{(\nu,\varrho)}$. This shows the non-commensurability.

Suppose now that $n_{0}=n$ and $m_{0}=m$. We prove that $L_{(\sigma,\rho)}/L_{(\sigma,\rho)}\cap L_{(\nu,\varrho)}$ is finite. Let $(x,t)\in L_{(\sigma,\rho)}$ with $t=t_{1}w_{1}+\ldots+t_{m}w_{m}\in\rho\mathbb{Z}^{m}$, where $(t_{1},\ldots,t_{m})\in\mathbb{Z}^{m}$. For each $i=1,\dots, m,$ write
$$
t_{i}=s_{i}+k_{i}q_{i},
$$
with $(s_{i},k_{i})\in\mathbb{N}\times\mathbb{Z}$ such that $0\leq s_{i}<q_{i}$. Then,
$$t=s+\kappa,
\quad\text{where}\quad
s=\sum_{j=1}^{m}s_{j}w_{j} \text { and } \kappa=\sum_{j=1}^{m}k_{j}q_{j}w_{j}.$$
Here, $s\in\rho\mathbb{Z}^{m}$ and $\kappa\in\rho\mathbb{Z}^{m}\cap\varrho\mathbb{Z}^{m}$.

Since $(0,-\kappa), (x,0)\in L_{(\sigma,\rho)}$, we have
$$(\eta(-\kappa)x,-\kappa)=(0,-\kappa)(x,0)\in L_{(\sigma,\rho)}.$$
Thus $\eta(-\kappa)x\in\sigma^{-1}\mathbb{Z}^{n}$, and we can write
$$\eta(-\kappa)x=\xi+\zeta,$$
where
$$
\xi=\sum_{i=1}^{n}r_{i}v_{i}\in\sigma^{-1}\mathbb{Z}^{n} \text { and } \zeta=\sum_{i=1}^{n}l_{i}p_{i}v_{i}\in\sigma^{-1}\mathbb{Z}^{n}\cap\nu^{-1}\mathbb{Z}^{n},
$$
with $r_{i}\in\mathbb{N}$, $0\leq r_{i}<p_{i}$ and $l_{i}\in\mathbb{Z}$.

Now, observe that $(\xi,s)\in L_{(\sigma,\rho)}$ and
\begin{align*}
(x,t)(\xi,s)^{-1}= &(x,t)(-\eta(-s)\xi,-s) \\
= & (x-\eta(t-s)\xi,t-s) \\
= & (x-\eta(\kappa)\xi,\kappa).
\end{align*}
Hence,
$$
(x,t)(\xi,s)^{-1}= (\eta(\kappa)\zeta,\kappa).
$$
Since the right-hand side belongs to $L_{(\sigma,\rho)}\cap L_{(\nu,\varrho)}$, it follows that
$$
\overline{(x,t)}=\overline{(\xi,s)}.
$$
This proves that
\begin{align*}
&L_{(\sigma,\rho)}/L_{(\sigma,\rho)}\cap L_{(\nu,\varrho)}=&\left\{\overline{\left(\sum_{i=1}^{n}r_{i}v_{i},\sum_{j=1}^{m}s_{j}w_{j}\right)},\
r_{i},s_{j}\in\mathbb{N},\ r_{i}<p_{i},\ s_{j}<q_{j}\right\}
\end{align*}
is finite. A similar argument shows that $L_{(\nu,\varrho)}/L_{(\sigma,\rho)}\cap L_{(\nu,\varrho)}$ is also finite. Therefore, the lattices $L_{(\sigma,\rho)}$ and $ L_{(\nu,\varrho)}$ are commensurable.

\end{proof}
\begin{lemma}\label{lem2}
Let $A\in GL_{n}(\mathbb{R})$. The following assertions are equivalent:
\begin{enumerate}[label=\roman*)]
\item $\mathrm{rank}_{\mathbb{Z}}(A\mathbb{Z}^{n}\cap\mathbb{Z}^{n})=n$
\item $A\in GL_{n}(\mathbb{Q})$
\end{enumerate}
\end{lemma}
\begin{proof}
Suppose that $\mathrm{rank}_{\mathbb{Z}}(A\mathbb{Z}^{n}\cap\mathbb{Z}^{n})=n$. Then there exists a matrix $B\in\mathbb{Z}^{n\times n}\cap GL_{n}(\mathbb{Q})$ such
$$A\mathbb{Z}^{n}\cap\mathbb{Z}^{n}=B\mathbb{Z}^{n}.$$
This implies $B\mathbb{Z}^{n}\subset A\mathbb{Z}^{n}$, and thus
$$A^{-1}B\mathbb{Z}^{n}\subset \mathbb{Z}^{n}.$$
Therefore $A^{-1}B\in GL_{n}(\mathbb{Q})$. Since $B\in GL_{n}(\mathbb{Q})$, it follows that $A\in GL_{n}(\mathbb{Q})$.

Conversely, suppose $A\in GL_{n}(\mathbb{Q})$. Then there exists a non zero integer $c(A)$ such that $c(A)A\in\mathbb{Z}^{n\times n}$. Furthermore, $\mathrm{rank}_{\mathbb{Z}}((c(A)A)\mathbb{Z}^{n})=n$ and since
$c(A)A\mathbb{Z}^{n}\subset\mathbb{Z}^{n}$, we have
$$\mathrm{rank}_{\mathbb{Z}}(A\mathbb{Z}^{n}\cap\mathbb{Z}^{n})=n.$$
\end{proof}
\begin{theorem}\label{Gcommensurable2}
Let $(\sigma ,\rho)$ and $(\nu,\varrho)$ be two $ G$-compatible pairs.
The lattices $L_{(\sigma ,\rho)}=\sigma^{-1}\mathbb{Z}^{n}\rtimes_\eta\rho\mathbb{Z}^{m}$ and $L_{(\nu,\varrho)}=\nu^{-1}\mathbb{Z}^{n}\rtimes_\eta\varrho\mathbb{Z}^{m}$ are commensurables if and only if $\nu\sigma^{-1}\in GL_{n}(\mathbb{Q})$ and $\varrho^{-1}\rho\in GL_{m}(\mathbb{Q})$.
\end{theorem}
\begin{proof}According to Proposition \ref{Gcommensurable1}, the lattices $L_{(\sigma ,\rho)}$ and $L_{(\nu,\varrho)}$ are commensurable if and only if
$$\mathrm{rank}_{\mathbb{Z}}(\sigma^{-1}\mathbb{Z}^{n}\cap\nu^{-1}\mathbb{Z}^{n})=n\quad\text{ and} \mathrm{rank}_{\mathbb{Z}}(\rho\mathbb{Z}^{m}\cap\varrho\mathbb{Z}^{m})=m.$$ This
means that $\mathrm{rank}_{\mathbb{Z}}(\nu\sigma^{-1}\mathbb{Z}^{n}\cap\mathbb{Z}^{n})=n$ and
$\mathrm{rank}_{\mathbb{Z}}(\varrho^{-1}\rho\mathbb{Z}^{m}\cap\mathbb{Z}^{m})=m$. According to Lemma \ref{lem2}, this equivalent that $\nu\sigma^{-1}\in GL_{n}(\mathbb{Q})$ and $\varrho^{-1}\rho\in GL_{m}(\mathbb{Q})$.
\end{proof}

\section{examples}
  \begin{itemize}
    \item [(a)] For $n=2$, consider the real symmetric matrix   $A=\left(
                                                      \begin{array}{cc}
                                                        2 & 1 \\
                                                        1 & 1 \\
                                                      \end{array}
                                                    \right)$. We can check that $A$ is definite positive hyperbolic matrix in $SL(2,\mathbb Z)$ and it exists $B\in\frak{sl}(2,\mathbb R)\cap GL(2,\mathbb R)$ such that $A=\exp(B)$.
Indeed, $A$ is diagonalizable over $\mathbb R$ with positive eigenvalues $$\lambda_1=\frac{3+\sqrt{5}}{2},\quad\lambda_2=\frac{3-\sqrt{5}}{2}.$$ Therefore, it exist $P\in GL_2(\mathbb R)$ and $D=diag(\lambda_1,\lambda_2)$ such that $A=PDP^{-1}$. For instance we can choose $P$ such that
$$
P=\left(
    \begin{array}{cc}
      \frac{1-\sqrt 5}{2} & \frac{1+\sqrt 5}{2} \\
      1 & 1 \\
    \end{array}
  \right).$$
Thus if we let $\Delta=diag( \ln\lambda_1,\ln\lambda_2)$, we obtain
$$
A=PDP^{-1}=Pe^{\Delta} P^{-1}.
$$ Thus the pair $(P,1)$ is $G-$compatible since $Pe^\Delta P^{-1}=A\in SL(2,\mathbb Z)$, and
$$L_{(P,1)}:=P^{-1}\mathbb Z^2\rtimes_\eta\mathbb Z$$ is a lattice in $G$ with
$$
 P^{-1}=\frac{1}{\sqrt 5}\left(
                                          \begin{array}{cc}
                                            -1 & \frac{1+\sqrt 5}{2} \\
                                            1 & -\frac{1-\sqrt 5}{2} \\
                                          \end{array}
                                        \right);\quad \eta(t)=e^{t\Delta}.$$

    \item [(b)] For $n=3$, let $A=\left(
                                            \begin{array}{ccc}
                                              k^2+1& 0 & k \\
                                              0 & 1 & l \\
                                              k& l & l^2+1 \\
                                            \end{array}
                                          \right)$ where $k, l\in\mathbb Z$.
                                          It is easy to check that   $A$ is definite positive matrix with $\det A=1$ and hence it is diagonalizable over $\mathbb R$ with positive eigenvalues
    $\lambda_1,\lambda_2, \lambda_3$. Therefore  it exist $P\in GL_3(\mathbb R)$ and a diagonal matrix $D$ with eigenvalues  $(\lambda_j)_{1\leq j\leq 3}$ such that $A=PDP^{-1}$. Now by letting $\Delta=\ln(D)=diag(\ln(\lambda_1),\ln(\lambda_2),\ln(\lambda_3))$. Thus if we let $G=\mathbb R^3\rtimes_\eta\mathbb R$ with $\eta(t)=e^{t\Delta}$, we see that the pair $(P,1)$ is $G-$compatible and  $L_{(P,1)}:=P^{-1}\mathbb Z^3\rtimes_\eta\mathbb Z$  is a lattice in $G$.

     \end{itemize}


\begin{thebibliography}{99}

\vskip 0.5 cm


\bibitem{AC} {\sc D. Arnal, B. Currey}, {\it Representations of Solvable Lie groups: Basic Theory and Examples}. Cambridge University Press.
DOI: 10.1017/9781108552288.\\


\bibitem{Au} {\sc L. Auslander}, {\it An exposition of the structure of solvmanifolds} – part I, Bull.
Amer. Math. Soc. 79 (1973), no. 2, 227–261.\\



\bibitem{BCDLRRV} {\sc P. Bernat, N. Conze, M. Duflo, M. L\'evy-Nahas, M. Rais, P. Renouard, M. Vergne}, {\it Repr\'esentations des groupes de Lie r\'esolubles}. Monographies de la S.M.F., Dunod (Paris) (1972).\\

\bibitem{CB} {\sc Christoph Bock}, {\it On Low-Dimensional Solvmanifolds}, Asian J. Math.
 Vol. 20, No. 2, pp. 199--262, April 2016.\\




\bibitem{CFT} {\sc B. Currey, H. F\"uhr, K.F. Taylor}, {\it Integrable wavelet transforms with abelian dilation groups}, Jour. Lie Th. {\bf 26} (2016), No. 2, 567--595.\\


\bibitem{Fer}{\sc M. Fern\'andez, V. Mu\~{n}oz}, {\it Formality of Donaldson submanifolds}, Math. Z.
250 (2005), no. 1, 149–175.\\


\bibitem{M} {\sc J. Milnor}, {\it Curvature of left invariant metrics on Lie groups, Advances in Math.}, 21:3 (1976), pp. 293–329\\


\bibitem{Mo} { \sc G. D. Mostow}, {\it Factor spaces of solvable groups}, Ann. of Math. (2), 60 (1954), pp. 1–27\\


\bibitem{MM} {\sc Richard Mosak, Martin Moskowitz}, {\it Lattices in split solvable Lie group}, Math. Proc. Camb. Phil. Soc. (1987), {\bf 122}, 145, 245--250.\\





\bibitem{MSR} {\sc M.S. Raghunathan}, {\it Discrete Subgroups of Lie groups}, Springer Verlag 1972.\\


\bibitem{Raw} {\sc Rawnsley J.H.}, { \it Representations of a semi direct product by quantization.} Math. Proc. Cambridge Philos. Soc. 78 (1975), 345--350.\\



\bibitem{OT}{\sc J. Oprea and A. Tralle}, {\it Symplectic Manifolds with no K¨ahler Structure}, Lecture Notes in
Math., 1661, Springer (1997).\\


\bibitem{T} {\sc Aleksy Tralle }, {\it On Solvable Lie Groups without Lattices}, Cont. Math., Vol. 288, 2001.\\

\bibitem{TsYama} {\sc Nobuo Tsuchiya, Aiko Yamakawa}, {\it Lattices of some solvable Lie groups and actions of products of affine groups.} Tohoku Math. J. 61 (2009), 349–364\\




\bibitem{VS} {\sc V. S. Vararadarajan}, {\it Lie Groups, Lie Algebras, and Their Representations}, Springer (1984).\\


\bibitem{VGS} {\sc E. B. Vinberg, V. V. Gorbatsevich and O. V. Schvartsman}, {\it Discrete subgroups of Lie groups, Lie
groups and Lie algebras II, (Eds. A. L. Onishchik and E. B. Vinberg)}, 1–123, Encyclopaedia Math. Sci. 21,
Springer-Verlag, New York-Berlin, 2000. \\


\bibitem{TY} {\sc Takumi Yamada}, {\it A construction of Lattices in solvable Lie groups}, Kodai Math. J. {\bf 39} (2016) 378--388.\\

\bibitem{Wang} {\sc H.C. Wang}, {\it Discrete subgroups of solvable Lie groups}, Ann. of Math. 64 (1956), 1--19.


\end{thebibliography}
\end{document}